\documentclass{fundam2}
\usepackage{url} 
\usepackage[ruled,lined]{algorithm2e}
\usepackage{graphicx}

\usepackage{amssymb}
\usepackage{amsfonts}
\usepackage{yfonts}

\renewcommand{\QED}{\hfill $\boxtimes\hspace{-1.725ex}\boxplus$}

\begin{document}
\setcounter{page}{1}
\issue{ }

\title{On  Axiomatizability of the Multiplicative Theory of Numbers}

\address{\sl \textsc{Saeed Salehi}, Research Institute for Fundamental Sciences \textup{(}RIFS\textup{)}, University of Tabriz,  P.O.Box~51666--16471, Bahman 29$^\text{th}$ Boulevard, Tabriz, IRAN.}

\author{
Saeed Salehi\thanks{
This research was partially  supported  by a  grant from $\mathbb{IPM}$ (No. 91030033).
}\\
Research Institute for Fundamental Sciences \textup{(RIFS)},
University of Tabriz, \\ P.O.Box~51666--16471,   Bahman 29$^\text{th}$ Boulevard, Tabriz, IRAN\\
School of Mathematics,  Institute for Research in Fundamental Sciences \textup{(}$\mathbb{IPM}$\textup{)},\\
P.O.Box 19395--5746,  Niavaran, Tehran, IRAN\\
http://saeedsalehi.ir/ \; salehipour{@}tabrizu.ac.ir  \, saeedsalehi{@}ipm.ir \\
} \maketitle

\runninghead{\copyright\ \textup{\sf 2017} \textsc{Saeed~Salehi}   }{  On Axiomatizability of the Multiplicative Theory of Numbers}

\begin{abstract}
The multiplicative theory of a set of numbers (which could be natural, integer, rational, real or complex numbers)  is the first-order  theory of the structure of that set with (solely) the multiplication operation (that set is taken to be multiplicative, i.e., closed under multiplication). In this paper we study the multiplicative theories  of the complex, real and (positive) rational numbers. These theories (and also the multiplicative theories of natural and integer numbers) are known to be decidable (i.e., there exists an algorithm that decides whether a given sentence is derivable form the theory); here we present  explicit axiomatizations for them and show that they are not finitely axiomatizable. For each of these sets (of complex, real and [positive] rational numbers) a language, including the multiplication operation, is introduced in a way that it allows quantifier elimination (for the theory of that set).
\end{abstract}


\section{Introduction}
Providing a (complete and computably decidable) axiomatization for  mathematical structures is one goal of mathematical logic. This is closely related to the problem of the (computable)  decidability of (the theory of) a given mathematical  structure, since by the (computable) enumerability of all the formulas (provided that the language of the structure is a computably decidable set), provability of a sentence or its unprovability (which is equivalent to the provability of its negation in complete axiomatizations) can be decided algorithmically in a finite number of steps. Thus by presenting a complete and computably  decidable axiomatization for a structure, the computable decidability of the theory of that structure is proved. While the mere knowledge of the decidability of  the theory of a structure does not provide us with an explicit axiomatization (for the theory of that structure) and also leaves open the problem of the finite axiomatizability of that structure.

In this paper we study the theories  of the sets of complex, real and (positive) rational numbers with the multiplication operation.
The multiplicative structure of the complex numbers, i.e., $\langle\mathbb{C};\times\rangle$,  is decidable (and completely and  computably axiomatizable) by Tarski's theorem which states that the (additive and multiplicative) theory of the  complex numbers ($\langle\mathbb{C};+,\times\rangle$) is decidable and can be (completely) axiomatized  by the theory of {\em algebraically closed fields of characteristic $0$} (see e.g.  \cite[Section~IV of Chapter~4]{kk} or \cite[Corollary~2.2.9]{marker} or \cite[Theorem~21.9]{monk}).
Here, we axiomatize the theory of this structure directly (without using Tarski's theorem) and show that it cannot be axiomatized  by any finite number of sentences. The same holds for the multiplicative theory of the real numbers, $\langle\mathbb{R};\times\rangle$: it is also decidable (and completely and  computably axiomatizable) by Tarski's theorem which states that the (additive and multiplicative) theory of the  real numbers ($\langle\mathbb{R};+,\times\rangle$) is decidable and can be (completely) axiomatized  by the theory of {\em real closed \textup{(}ordered\textup{)} fields} (see e.g.  \cite[Section~V of Chapter~4]{kk} or \cite[Corollary~3.3.16]{marker} or \cite[Theorem~21.36]{monk}). Again an explicit axiomatization for the theory of this structure is provided here, in which the addition operation is not used.

The decidability of the multiplicative structure of the non-zero  rational numbers was announced by A.~Mostowski in \cite{mostowski} where he mentions that ``the elementary theory of multiplication of rationals different from $0$'' is the weak power of the additive theory of ``all integers (positive and negative)''. Here, ``the elementary theory'' means `the first-order theory'. The decidability of this theory is claimed to had been  proved (beforehand) by W.~Szmielew in \cite{szmielew}.
We firstly note that the weak power of the additive theory of integers, $\langle\mathbb{Z};+\rangle$, is the multiplicative theory of the {\em positive} rational numbers, $\langle\mathbb{Q}^+;\times\rangle$; not the whole (non-zero) rational numbers. Secondly, the multiplicative theory of the positive rational numbers has not been studied in \cite{szmielew} (indeed it appears in none of Szmielew's works). However, Mostowski's results in \cite{mostowski} imply the decidability of the multiplicative theory of the positive rational numbers $\langle\mathbb{Q}^+;\times\rangle$. In the last section of this paper, we give a direct proof of this fact with an explicit axiomatization, and show that the theory of this structure is not finitely axiomatizable. Along the way, for technical reasons, we also study the additive theories  of the sets of integer, rational, real and complex numbers, $\langle\mathbb{Z};+\rangle$, $\langle\mathbb{Q};+\rangle$, $\langle\mathbb{R};+\rangle$ and $\langle\mathbb{C};+\rangle$ (for $\langle\mathbb{N};+\rangle$ see e.g.~\cite[Theorem~32A]{enderton}).
The present paper is an extended, much improved and corrected version of the conference paper~\cite{salehi}.

\subsection{Some Preliminaries}
Let $\mathbb{P}$ denote the set of all (natural) prime numbers and denote the $i^{\rm th}$  prime number by $\mathfrak{p}_i$  (so we have   $\mathfrak{p}_0=2,\mathfrak{p}_1=3,\mathfrak{p}_2=5,\cdots$). Every natural number other than $0,1$ has a unique factorization into a product of prime numbers (by the fundamental theorem of arithmetic); this holds for every negative integer other than $-1$ too. Likewise, every rational number other than $-1,0,1$ has a unique factorization into a product of prime numbers in which the exponents could be negative; for example $\frac{175}{84}$ which becomes $\frac{25}{12}$ after simplification can be written as $2^{-2}\cdot 3^{-1} \cdot 5^2$. The symbols $\times$ and $\cdot$ are used interchangeably throughout the paper. For convenience, we make the convention  that $0^{-1}=0$ and we will see that this does not contradict our intuition with the axioms used below. Needless to say $x^n$ symbolizes $x\cdot x \cdot \ldots \cdot x$ ($n-$times) and also $x+x+\cdots +x$ ($n-$times) is abbreviated as $n\centerdot x$.
The main tool for the process of quantifier elimination is the following result which can be found in e.g. \cite[Theorem~31F]{enderton} or
\cite[Theorem~1 in Chapter~4]{kk} or \cite[Lemma~3.1.5]{marker} or  \cite[Lemma~4.1]{smorynski}.

\begin{lemma}[The Main Lemma of Quantifier Elimination]\label{mainlem}
A theory (or a structure) admits quantifier elimination if and only if every formula of the form $\exists x (\bigwedge\hspace{-1.85ex}\bigwedge_{i}\alpha_i)$ is (recursively) equivalent with  a quantifier-free formula, where each $\alpha_i$ is either an atomic formula or the negation of an atomic formula.
\end{lemma}
\begin{proof}
Every formula $\psi$ can be written (equivalently) in the prenex normal form, say $${\sf Q}_1x_1{\sf Q}_2x_2\cdots{\sf Q}_nx_n\theta(x_1,x_2,\cdots,x_n),$$ where ${\sf Q}_i$'s are quantifiers and $\theta$ is quantifier-free. If ${\sf Q}_n=\exists$ then let $\theta'=\theta$ and if ${\sf Q}_n=\forall$ then let $\theta'=\neg\theta$ (note that in the latter case $\forall x_n\theta \equiv \neg\exists x_n\theta'$). Now, the quantifier-free formula $\theta'$ can be written in the disjunctive normal form, say $\bigvee\hspace{-1.85ex}\bigvee_i
\bigwedge\hspace{-1.85ex}\bigwedge_{j}\alpha_{i,j}$ where each $\alpha_{i,j}$ is a literal (i.e., an atomic or a negated atomic formula). Noting that $\exists x (\bigvee\hspace{-1.85ex}\bigvee_i\beta_i)\equiv
\bigvee\hspace{-1.85ex}\bigvee_i
\exists x \beta_i$  we have $$\psi \equiv {\sf Q}_1x_1{\sf Q}_2x_2\cdots{\sf Q}_{n-1}x_{n-1}\Box\bigvee\hspace{-2.5ex}\bigvee_i\exists x_n(\bigwedge\hspace{-2.5ex}\bigwedge_j\alpha_{i,j})$$ where $\Box$ is nothing (empty) when ${\sf Q}_n=\exists$ and $\Box=\neg$ when ${\sf Q}_n=\forall$. Now, if $\exists x_n(\bigwedge\hspace{-1.85ex}\bigwedge_j\alpha_{i,j})$ is equivalent with a quantifier-free formula, then $\psi$ is equivalent with a formula with one less quantifier; continuing this way one can show that $\psi$ is equivalent with a formula which has no quantifier.
\end{proof}

\section{The Multiplicative Theory of the Complex Numbers}
For axiomatizing $\langle\mathbb{C};\times\rangle$ we do not need the addition operation ($+$) and in fact  it is not definable from multiplication: the multiplicative automorphism $z\mapsto z^{-1}$ (for $z\neq 0$ and $0\mapsto 0$) does not preserve the addition operation. Indeed, there exists a nice axiomatization for the multiplicative theory of the complex numbers which will be presented below.
\begin{definition}[Roots of Unity]\label{root1}
For any natural number $n\!\geqslant\!2$ let $\boldsymbol\omega_n=\cos(2\boldsymbol\pi/n)
+\boldsymbol\imath^{\!^{\boldsymbol\prime}}
\sin(2\boldsymbol\pi/n)$. So, all the $n^{\rm th}$ roots of unity are the complex numbers
$\{1,\boldsymbol\omega_n,\boldsymbol\omega_n^2,
\cdots,
\boldsymbol\omega_n^{n-1}\}$. 
\hfill $\oplus\hspace{-1.725ex}\otimes$
\end{definition}
Let us note that $\boldsymbol\omega_2\!=\!-1, \boldsymbol\omega_3\!=\!(-1/2) + \boldsymbol\imath^{\!^{\boldsymbol\prime}} (\sqrt{3}/2)$ and  $\boldsymbol\omega_4\!=\!\boldsymbol\imath^{\!^{\boldsymbol\prime}}$.
\begin{theorem}[Infinite Axiomatizablity of $\langle\mathbb{C};\times\rangle$]\label{thm-c}
The following theory
completely axiomatizes the multiplicative theory of the complex numbers and, moreover, the infinite structure $\langle\mathbb{C};\times,\circ^{-1},{\bf 0},{\bf 1},
\boldsymbol\omega_2,\boldsymbol\omega_3,\boldsymbol\omega_4,
\cdots\rangle$ admits quantifier elimination, and so has a decidable theory.
\begin{center}
\begin{tabular}{ll}
($\texttt{M}_1$) \; $\forall x,y,z\,\big(x\cdot(y\cdot
z)=(x\cdot y)\cdot z\big)$  & ($\texttt{M}_2$) \; $\forall x\,\big(x\cdot \mathbf{1}=x\big)$ \\
($\texttt{M}_3$) \; $\forall x\,\big(x\neq \mathbf{0}\longrightarrow
 x\cdot x^{-1}=\mathbf{1}\big)$ & ($\texttt{M}_4$) \; $\forall x,y\,\big(x\cdot y=y\cdot x\big)$ \\
 ($\texttt{M}_5$) \; $\forall x\,\big(x\cdot \mathbf{0}=\mathbf{0}=\mathbf{0}^{-1}\big)$ & ($\texttt{M}_{6,n}$) \; $\bigwedge\hspace{-1.85ex}\bigwedge_{i<j<n}
 (\boldsymbol\omega_n)^i\neq (\boldsymbol\omega_n)^j$ \\
($\texttt{M}_{7,n}$) \; $\forall
x\,\big(x^n=\mathbf{1}
\longleftrightarrow\bigvee\hspace{-1.85ex}\bigvee_{i<n}x=
(\boldsymbol\omega_n)^i\big)$ & ($\texttt{M}_{8,n}$) \; $\forall x \, \exists y \, \big(y^n=x\big)$ \\
\end{tabular}
\end{center}
Where $n>1$ is a natural number.
\end{theorem}
\begin{proof}
By $\texttt{M}_1,\texttt{M}_2,\texttt{M}_3$ and $\texttt{M}_4$ we have $(u\cdot v)^{-1}=u^{-1}\cdot v^{-1}$ for any $u,v\neq{\bf 0}$; by $\texttt{M}_5$   this holds even when any of $u$ or $v$ equals to ${\bf 0}$. So, every term involving $x$ is equal to $x^{k}\cdot t$ for some $x$-free term $t$ (i.e., $x$ does not appear in $t$) and for some $k\in\mathbb{Z}-\{{\bf 0}\}$. Therefore, every atomic formula involving $x$ is equivalent with  $x^k\cdot t=x^m\cdot u$ for some $x$-free terms $t,u$  and some $\langle k,m\rangle \in\mathbb{N}^2-\{\langle 0,0\rangle\}$. If  $k\geqslant m>0$ then this atomic formula is equivalent with $$(x={\bf 0})\vee (x\neq{\bf 0}\wedge t={\bf 0}\wedge u={\bf 0})\vee (x\neq{\bf 0}\wedge t\neq{\bf 0}\wedge x^{k-m}=u\cdot t^{-1}),$$
and if $k\geqslant m=0$ then it is equivalent with $$(t={\bf 0}\wedge u={\bf 0})\vee(t\neq{\bf 0}\wedge x^k=u\cdot t^{-1}).$$
Also, the negated atomic formula $x^k\cdot t\neq x^m\cdot u$, when $k\geqslant m>0$, is equivalent with $$(x\neq{\bf 0}\wedge t={\bf 0}\wedge u\neq {\bf 0})\vee (x\neq{\bf 0}\wedge t\neq{\bf 0}\wedge x^{k-m}\neq u\cdot t^{-1}),$$ and when $k\geqslant m=0$ is equivalent with $$(t={\bf 0}\wedge u\neq{\bf 0})\vee(t\neq{\bf 0}\wedge x^k\neq u\cdot t^{-1}).$$
So, by the Main Lemma (\ref{mainlem}) it suffices to show that every formula of the form
\begin{equation}\label{c-1}
\exists x ( \bigwedge\hspace{-2.5ex}\bigwedge_{i<\ell} x^{n_i}=t_i \;\wedge\;  \bigwedge\hspace{-2.7ex}\bigwedge_{j<k} x^{m_j}\neq s_j )
\end{equation}
is equivalent with a quantifier-free formula, where $t_i$'s and $s_j$'s are $x$-free terms and $n_i$'s and $m_j$'s are positive natural numbers.
If $\ell=0$ then the formula
\eqref{c-1}, that is $\exists x (\bigwedge\hspace{-1.85ex}\bigwedge_{j<k} x^{m_j}\neq s_j)$, follows from $\texttt{M}_{6,n}$ and $\texttt{M}_{7,n}$  (and so it is equivalent with the quantifier-free formula ${\bf 0}={\bf 0}$): by $\texttt{M}_{6,n}$'s there are infinitely many elements and by $\texttt{M}_{7,n}$ (for $n=m_j$) there are at most finitely many $x$'s with $x^{m_j}=s_j$ for each $j<k$. Whence, let us suppose that $\ell>0$. If there are some $i,j<\ell$ such that $n_i<n_j$ then $(x^{n_i}\!=\!t_i\wedge x^{n_j}\!=\!t_j) \equiv (t_i\!=\!{\bf 0}\wedge x\!=\!{\bf 0} \wedge t_j\!=\!{\bf 0})\vee(t_i\neq {\bf 0}\wedge x^{n_i}\!=\!t_i\wedge x^{n_j-n_i}\!=\!t_j\cdot t_i^{-1})$. So, we can assume that for some $n>0$ we have $n_i=n$ for all $i<\ell$. Then, for $t=t_0$, the formula~\eqref{c-1}  is equivalent with the conjunction of the formula $\bigwedge\hspace{-1.85ex}\bigwedge_{i<\ell}t_i=t$ with the following formula
\begin{equation}\label{c-2}
\exists x ( x^{n}=t \;\wedge\;  \bigwedge\hspace{-2.7ex}\bigwedge_{j<k} x^{m_j}\neq s_j )
\end{equation}
whose equivalence with a quantifier-free formula is proved below. Let us note that if $k=0$ then \eqref{c-2} follows from $\texttt{M}_{8,n}$ (and so is equivalent with ${\bf 0}={\bf 0}$). Whence, let us suppose that $k>0$. By  $\texttt{M}_{6,n}$ and  $\texttt{M}_{7,n}$ we have the equivalence  $x^m\neq s \longleftrightarrow x^{mn}\neq s^n\vee
\bigvee\hspace{-1.85ex}\bigvee_{0<i<n}x^m=s(\boldsymbol\omega_n)^i$  for all complex numbers $x,s$ with $s\neq 0$ and all  natural numbers $m,n$. Thus,  $\theta \wedge x^n=t \wedge x^m\neq s$ is equivalent with
$\big[s={\bf 0}\wedge \theta \wedge x^n=t \wedge x\neq{\bf 0}\big] \vee \big[s\neq{\bf 0} \wedge \theta \wedge x^n=t \wedge \big(x^{mn}\neq s^n\vee
\bigvee\hspace{-1.85ex}\bigvee_{0<i<n}x^m=s
(\boldsymbol\omega_n)^i\big)\big]$.
The second disjunct is equivalent with

\medskip

\begin{tabular}{rc}
$\big[s\neq{\bf 0} \wedge \theta \wedge x^n=t \wedge \big(x^{mn}\neq s^n\vee
\bigvee\hspace{-1.85ex}\bigvee_{0<i<n}x^m=s
(\boldsymbol\omega_n)^i\big)\big]$ & $\equiv$ \\
$(s\neq{\bf 0} \wedge \theta \wedge x^n=t \wedge x^{mn}\neq s^n) \vee
\bigvee\hspace{-1.85ex}\bigvee_{0<i<n}\big(s\neq{\bf 0} \wedge \theta \wedge x^n=t \wedge x^m=s(\boldsymbol\omega_n)^i\big)$ & $\equiv$ \\
$(s\neq{\bf 0} \wedge \theta \wedge x^n=t \wedge t^m\neq s^n) \vee
\bigvee\hspace{-1.85ex}\bigvee_{0<i<n}\big(s\neq{\bf 0} \wedge \theta \wedge x^n=t \wedge x^m=s(\boldsymbol\omega_n)^i\big)$.  &
\end{tabular}

\medskip

\noindent Continuing this way (by eliminating the inequalities ---other than $x\neq{\bf 0}$--- one by one) we see that all we need to do is to eliminate the quantifier of the following form of formulas
\begin{equation}\label{c-3}
\exists x (\bigwedge\hspace{-2.5ex}\bigwedge_{i<\ell} x^{n_i}=t_i) \textrm{\quad or \quad} \exists x (x\neq{\bf 0}\wedge \bigwedge\hspace{-2.5ex}\bigwedge_{i<\ell} x^{n_i}=t_i)
\end{equation}
where $n_i$'s are positive natural numbers and $t_i$'s are $x$-free terms (i.e., $x$ does not appear in them). Just like the way we reached at \eqref{c-2} from \eqref{c-1} we can also see that the formulas \eqref{c-3} are equivalent with the conjunctions of an $x$-free formula with a formula of the form $\exists x (x^n=t)$ or $\exists x (x\neq{\bf 0}\wedge x^n=t)$ for some positive integer $n$ and some $x$-free term $t$. Above we noted that $\exists x (x^n=t)$ follows from $\texttt{M}_{8,n}$ and so is equivalent with the quantifier-free formula ${\bf 0}={\bf 0}$; thus $\exists x (x\neq{\bf 0}\wedge x^n=t)$ is equivalent with the quantifier-free formula $t\neq{\bf 0}$ as well.
\end{proof}

Now we show that the multiplicative theory of the complex numbers is not finitely axiomatizable.

\begin{definition}[Two Additive Sub-Structures of the Rational Numbers]\label{def-a}
Let $m\in\mathbb{N}$ be a positive integer ($m>0$). Put
$$\mathbb{Z}/m =\{\dfrac{a}{m} \mid a\in\mathbb{Z}\}, \quad  \textrm{and}$$
$$\mathbb{Q}/m =\{\dfrac{a}{m^k} \mid a\in\mathbb{Z},k\in\mathbb{N}\}.$$
These are additive (i.e., closed under addition) subsets of $\mathbb{Q}$.
\hfill $\oplus\hspace{-1.725ex}\otimes$
\end{definition}

Let us note that $\mathbb{Q}/m$ is also closed under the operations $x\mapsto x/d$ for any  $d$ which divides  $m$.

\begin{theorem}[No Finite Axiomatization for  $\langle\mathbb{C};\times\rangle$]\label{thm-infc}
The theory of the structure $\langle\mathbb{C};\times\rangle$ is not finitely axiomatizable.
\end{theorem}
\begin{proof}
If the theory is finitely axiomatizable by, say, $B_1,\cdots,B_k$ then the formula  $B_1\wedge\cdots\wedge B_k$ is provable from  $\{\texttt{M}_1,\texttt{M}_2,\texttt{M}_3,\texttt{M}_4,
\texttt{M}_5\}\cup\{\texttt{M}_{6,n},\texttt{M}_{7,n},
\texttt{M}_{8,n}\mid n>1\}$ hence  just a finite number of the instances of $\texttt{M}_{8,n}$ are used in the proof.
Let $N$ be an arbitrarily large natural number, and put $M=N!=2\times\cdots\times N$. Let
$$\mathbb{C}/M=\{\prod_{i<\ell}\mathfrak{p}_i^{r_i}
\prod_{j<k}(\boldsymbol\omega_j)^{n_j}\mid \ell,k,n_j\in\mathbb{N},r_i\in\mathbb{Q}/M\}.$$
The set $\mathbb{C}/M$ is a multiplicative   subset of $\mathbb{C}$ (is  closed under multiplication and inverses) and so satisfies 
$\texttt{M}_1,\texttt{M}_2,\texttt{M}_3,\texttt{M}_4,
\texttt{M}_5,\texttt{M}_{6,n}$ and $\texttt{M}_{7,n}$ (for all $n\!\geqslant\!2$). Since the set $\mathbb{Q}/M$ is closed under the operations $x\mapsto x/n$ for every $n\in\{1,2,3,\cdots,N\}$ then $\mathbb{C}/M$ satisfies 
$\texttt{M}_{8,n}\!:\!\forall x\exists y (y^n=x)$ for $n=2,3,\cdots,N$. But for a large prime number $\mathfrak{p}>M$ the structure $\langle\mathbb{C}/M;\times\rangle$ does not satisfy $\texttt{M}_{8,\mathfrak{p}}\!:\!\forall x\exists y (y^\mathfrak{p}=x)$ since by $1/\mathfrak{p}\not\in\mathbb{Q}/M$ we have $2^{1/\mathfrak{p}}\not\in\mathbb{C}/M$. So, the instances of  $\texttt{M}_{8,n}$ for $n=2,\cdots,N$ (together with the axioms $\texttt{M}_1,\texttt{M}_2,\texttt{M}_3,\texttt{M}_4,
\texttt{M}_5,\texttt{M}_{6,n}$ and $\texttt{M}_{7,n}$ for all $n\!>\!1$) does  not imply the instance of $\texttt{M}_{8,n}$ for $n=\mathfrak{p}$, where $\mathfrak{p}$ is a prime number greater than $N!$.
\end{proof}

\subsection{The Additive Theory of the Complex (and Real and Rational) Numbers}
It is interesting to have a look at the additive theory of the complex numbers (i.e., $\langle\mathbb{C};+\rangle$): its theory is the same as of the additive theory of the real and the rational numbers ($\langle\mathbb{R};+\rangle$ and $\langle\mathbb{Q};+\rangle$) and also the multiplicative theory of the positive real numbers ($\langle\mathbb{R}^+;\times\rangle$); cf. \cite[Theorem~3.1.9]{marker}.

\begin{proposition}[Infinite Axiomatizablity of $\langle\mathbb{C};+\rangle$ and $\langle\mathbb{R};+\rangle$ and $\langle\mathbb{Q};+\rangle$]\label{thm-a}
The following theory
completely axiomatizes the additive theory of the complex (and real and rational) numbers and, moreover, the structure $\langle\mathbb{C};+,-,{\bf 0}\rangle$ (and $\langle\mathbb{R};+,-,{\bf 0}\rangle$ and $\langle\mathbb{Q};+,-,{\bf 0}\rangle$) admits quantifier elimination, and so has a decidable theory.
\begin{center}
\begin{tabular}{ll}
($\texttt{A}_1$) \; $\forall x,y,z\,\big(x+(y+z)=(x+y)+z\big)$  & ($\texttt{A}_2$) \; $\forall x\,\big(x+\mathbf{0}=x\big)$ \\
($\texttt{A}_3$) \; $\forall x\,\big(
 x+ (-x)=\mathbf{0}\big)$ & ($\texttt{A}_4$) \; $\forall x,y\,\big(x+y=y+x\big)$ \\
 ($\texttt{A}_{5,n}$) \; $\forall
x\,\big(n\centerdot x={\bf 0}\longrightarrow x={\bf 0}\big)$ & ($\texttt{A}_6$) \; $\exists y \, \big(y\neq {\bf 0}\big)$ \\
($\texttt{A}_{7,n}$) \; $\forall x\exists y\,\big(x=n\centerdot y\big)$ & Where $n\geqslant 1$ is a natural number. \\
\end{tabular}
\end{center}
\end{proposition}
\begin{proof}
By $\texttt{A}_1,\texttt{A}_2,\texttt{A}_3$ and $\texttt{A}_4$ every term involving $x$ is equal to  $k\centerdot x + t$ for some $x$-free term $t$  and  $k\in\mathbb{Z}-\{{\bf 0}\}$. Therefore, every atomic formula involving $x$ is  equivalent with $k\centerdot x =t$ for some positive integer $k$ and some $x$-free term $t$. Thus, by the Main Lemma (\ref{mainlem}) it suffices to eliminate the quantifier of
\begin{equation}\label{a-1}
\exists x ( \bigwedge\hspace{-2.5ex}\bigwedge_{i<\ell} n_i\centerdot x=t_i \;\wedge\;  \bigwedge\hspace{-2.7ex}\bigwedge_{j<k} m_j\centerdot x\neq s_j).
\end{equation}
By $\texttt{A}_{5,k}$ (and $\texttt{A}_3$) we have $a=b \longleftrightarrow k\centerdot a = k\centerdot b$, and so we can assume that all $n_i$'s and all $m_j$'s in the formula \eqref{a-1} are equal to each other. Thus we show the equivalence of
\begin{equation}\label{a-2}
\exists x ( \bigwedge\hspace{-2.5ex}\bigwedge_{i<\ell} q\centerdot x=t_i \;\wedge\;  \bigwedge\hspace{-2.7ex}\bigwedge_{j<k} q\centerdot x\neq s_j)
\end{equation}
with a quantifier-free formula. By $\texttt{A}_{7,q}$, \eqref{a-2} is equivalent with the following formula (for $y=q\centerdot x$):
\begin{equation}\label{a-3}
\exists y (\bigwedge\hspace{-2.5ex}\bigwedge_{i<\ell} y=t_i \;\wedge\;  \bigwedge\hspace{-2.7ex}\bigwedge_{j<k} y\neq s_j).
\end{equation}
Now, if $\ell>0$ then \eqref{a-3} is equivalent with the quantifier-free formula $\bigwedge\hspace{-1.85ex}\bigwedge_{i<\ell} t_0=t_i \wedge  \bigwedge\hspace{-1.85ex}\bigwedge_{j<k} t_0\neq s_j$ and if $\ell=0$ then \eqref{a-3} i.e., the formula $\exists y (\bigwedge\hspace{-1.85ex}\bigwedge_{j<k} y\neq s_j)$ follows from $\texttt{A}_6$ (which together with $\texttt{A}_{5,n}$'s  implies that there are infinitely many elements: for $y\neq {\bf 0}$ we have  $k\centerdot y\neq m\centerdot y$ whenever $k\neq m$), and so is equivalent with the quantifier-free formula ${\bf 0}={\bf 0}$.
\end{proof}

Just a little note that this axiomatization of the additive theory of the complex, real and rational numbers cannot be finite:

\begin{proposition}[No Finite Axiomatization for  $\langle\mathbb{C};+\rangle$ and $\langle\mathbb{R};+\rangle$ and $\langle\mathbb{Q};+\rangle$]\label{thm-infa}
The theories of the structures $\langle\mathbb{C};+\rangle$,  $\langle\mathbb{R};+\rangle$ and $\langle\mathbb{Q};+\rangle$ are not finitely axiomatizable.
\end{proposition}
\begin{proof}
It suffices to note that $\texttt{A}_1,\texttt{A}_2,\texttt{A}_3,\texttt{A}_4,
\texttt{A}_{5,n},\texttt{A}_6,$ and a finite number of the instances of $\texttt{A}_{7,n}$ do not imply all the instances of $\texttt{A}_{7,n}$. For an arbitrary large $N$ let $M=N!=2\times\cdots\times N$. Then $\mathbb{Q}/M$ satisfies $\texttt{A}_1,\texttt{A}_2,\texttt{A}_3,\texttt{A}_4,
\texttt{A}_{5,n},\texttt{A}_6,$ and also $\texttt{A}_{7,n}: \forall x\exists y\,\big(x=n\centerdot y\big)$ for $n\in\{2,\cdots,N\}$, but does not satisfy the instance $\forall x\exists y\,\big(x=\mathfrak{p}\centerdot y\big)$ of $\texttt{A}_{7,\mathfrak{p}}$ for a large prime $\mathfrak{p}>M$.
\end{proof}

\section{The Multiplicative Theory of the Real Numbers}
The mapping  $x\mapsto 2^x$ is an isomorphism between the additive structure of the real numbers $\langle\mathbb{R};+\rangle$ and the multiplicative structure  of the positive reals $\langle\mathbb{R}^+;\times\rangle$. Indeed the proof of Proposition~\ref{thm-a} can show the (computable) axiomatizability (and decidability) of the theory of $\langle\mathbb{R}^+;\times\rangle$:
\begin{proposition}[Axiomatizablity of $\langle\mathbb{R}^+;\times\rangle$---Infinitely]\label{thm-r+}
The following theory
completely axiomatizes  the structure $\langle\mathbb{R}^+;\times,\circ^{-1},{\bf 1}\rangle$ and, moreover, its theory admits quantifier elimination, and so is decidable.
\begin{center}
\begin{tabular}{ll}
($\texttt{M}_1$) \; $\forall x,y,z\,\big(x\cdot (y\cdot z)=(x\cdot y)\cdot z\big)$  & ($\texttt{M}_2$) \; $\forall x\,\big(x\cdot \mathbf{1}=x\big)$ \\
($\texttt{M}_3^\circ$) \; $\forall x\,\big(x \cdot  x^{-1}=\mathbf{1}\big)$ & ($\texttt{M}_4$) \; $\forall x,y\,\big(x\cdot y=y\cdot x\big)$ \\
 ($\texttt{M}_{7,n}^\circ$) \; $\forall
x\,\big(x^n={\bf 1}\longrightarrow x={\bf 1}\big)$ &
($\texttt{M}_{8,n}$) \; $\forall x\exists y\,\big(x=y^n\big)$
 \\
($\texttt{M}_9$) \; $\exists y \, \big(y\neq {\bf 1}\big)$ &
Where $n\geqslant 1$ is a natural number.
\end{tabular}
\end{center}
However, this theory is not finitely axiomatizable.
\end{proposition}
\begin{proof}
For $M=N!$ the multiplicative subset of positive real numbers  $$\mathbb{R}^+/M=\{\prod_{i<\ell}\mathfrak{p}_i^{r_i}\mid \ell\in\mathbb{N},r_i\in\mathbb{Q}/M\}$$
satisfies $\texttt{M}_1,\texttt{M}_2,\texttt{M}_3^\circ,\texttt{M}_4,
\texttt{M}_{7,n}^\circ,\texttt{M}_9$ (for all $n\!\geqslant\!1$) and the instances of $\texttt{M}_{8,n}$ for $n=1,2,\cdots,N$ but does not satisfy the instance of $\texttt{M}_{8,n}$ when $n$ is a prime number greater than $M$.
\end{proof}
Let us note that $\langle\mathbb{R}^+;\times,\circ^{-1},{\bf 1}\rangle$ is an abelian group, and the theory of all abelian groups is decidable (proved by Szmielew for the first time in~\cite{szmielew55}).

Adding a zero to the elements with the axiom $\forall x\,\big(x\cdot \mathbf{0}=\mathbf{0}=\mathbf{0}^{-1}\big)$ can completely axiomatize the multiplicative theory of the non-negative real numbers $\langle\mathbb{R}^{\geqslant 0};\times\rangle$.
 Since the proof of the following theorem will be essentially repeated in Theorem~\ref{thm-r}, we do not present it.
\begin{proposition}[Infinite Axiomatizablity of $\langle\mathbb{R}^{\geqslant 0};\times\rangle$]\label{thm-r0}
The following theory
completely axiomatizes  the structure $\langle\mathbb{R}^{\geqslant 0};\times,\circ^{-1},{\bf 0},{\bf 1}\rangle$ and, moreover, its theory admits quantifier elimination, and so is decidable.
\begin{center}
\begin{tabular}{ll}
($\texttt{M}_1$) \; $\forall x,y,z\,\big(x\cdot (y\cdot z)=(x\cdot y)\cdot z\big)$  & ($\texttt{M}_2$) \; $\forall x\,\big(x\cdot \mathbf{1}=x\big)$ \\
($\texttt{M}_3$) \; $\forall x\,\big(x\neq{\bf 0}\longrightarrow x \cdot  x^{-1}=\mathbf{1}\big)$ & ($\texttt{M}_4$) \; $\forall x,y\,\big(x\cdot y=y\cdot x\big)$ \\
 ($\texttt{M}_{7,n}^\circ$) \; $\forall
x\,\big(x^n={\bf 1}\longrightarrow x={\bf 1}\big)$ &
($\texttt{M}_{8,n}$) \; $\forall x\exists y\,\big(x=y^n\big)$
 \\
($\texttt{M}_9^\circ$) \; $\exists y \, \big(y\neq {\bf 0}, {\bf 1}\big)$ & ($\texttt{M}_{10}$) \; $\forall x\,\big(x\cdot \mathbf{0}=\mathbf{0}=\mathbf{0}^{-1}\big)$ \\
Where $n\geqslant 1$ is a natural number. &
\end{tabular}

\end{center}
This theory is not finitely axiomatizable.
\QED
\end{proposition}
The whole set of the real numbers with the multiplication operation, i.e., the structure $\langle\mathbb{R};\times,\circ^{-1},{\bf 0},{\bf 1}\rangle$,  does not admit quantifier elimination: the formula $\exists x (y=x\cdot x)$ is not equivalent with any quantifier-free formula (in the language $\{\times,\circ^{-1},{\bf 0},{\bf 1},{\bf -1}\}$). Indeed this formula is equivalent with the quantifier-free formula $y\!\geqslant\!0$, so it is tempting to add order to the language for eliminating the quantifiers. But order is not definable by multiplication in $\mathbb{R}$ since the multiplicative automorphism $x\mapsto 1/x$ (for $x\neq 0$ and $0\mapsto 0$) does not preserve the order relation (neither does it preserve the addition operation). But if we add the {\em positivity}  property to the language, $\mathcal{P}(y)$ meaning that ``$y$ is a positive real number''  then the procedure of quantifier elimination can go through (then for example $\exists x (y\!=\!x\!\cdot\!x)$ is equivalent with the quantifier-free formula $\mathcal{P}(y)\!\vee\!y\!=\!{\bf 0}$). Below, $-x$ is a shorthand for $(-1)\cdot x$.
\begin{theorem}[Infinite Axiomatizablity of $\langle\mathbb{R};\times\rangle$]\label{thm-r}
The following theory
completely axiomatizes  the structure $\langle\mathbb{R};\times,\circ^{-1},{\bf 0},{\bf 1},{\bf -1},\mathcal{P}\rangle$ and, moreover, its theory admits quantifier elimination, and so is decidable.
\begin{center}
\begin{tabular}{ll}
($\texttt{M}_1$) \; $\forall x,y,z\,\big(x\cdot (y\cdot z)=(x\cdot y)\cdot z\big)$  & ($\texttt{M}_2$) \; $\forall x\,\big(x\cdot \mathbf{1}=x\big)$ \\
($\texttt{M}_3$) \; $\forall x\,\big(x\neq{\bf 0}\longrightarrow x \cdot  x^{-1}=\mathbf{1}\big)$ & ($\texttt{M}_4$) \; $\forall x,y\,\big(x\cdot y=y\cdot x\big)$ \\
($\texttt{M}_9^\diamond$) \; $\exists y \, \big(y\neq  {\bf -1}, {\bf 0}, {\bf 1}\big)$ & ($\texttt{M}_{10}$) \; $\forall x\,\big(x\cdot \mathbf{0}=\mathbf{0}=\mathbf{0}^{-1}\big)$ \\
 ($\texttt{M}_{11,n}$) \; $\forall
x\,\big(x^{2n}={\bf 1}\longleftrightarrow x={\bf 1}\vee x={\bf -1}\big)$ & ($\texttt{M}_{12,n}$) \; $\forall x\exists y\,\big(x=y^{2n+1}\big)$  \\
($\texttt{M}_{13}$) \; $\forall x\,\big(\mathcal{P}(x)\longleftrightarrow \exists y [y\neq{\bf 0}\wedge x=y^2]\big)$  & ($\texttt{M}_{14}$) \; $\forall x\,\big(x\neq{\bf 0}\longrightarrow [\neg\mathcal{P}(x)\leftrightarrow\mathcal{P}(-x)]\big)$ \\
($\texttt{M}_{15}$) \; $\forall x,y\neq{\bf 0}\,\big(\mathcal{P}(x\cdot y)\longleftrightarrow [\mathcal{P}(x)\leftrightarrow\mathcal{P}(y)]\big)$ &
Where $n \geqslant 1$ is a natural number.
\end{tabular}
\end{center}
\end{theorem}
\begin{proof}
Firstly, let us derive ($\texttt{M}_{16}$)  $\neg\mathcal{P}({\bf 0})\wedge \mathcal{P}({\bf 1})$  as follows: from $\texttt{M}_{2}$, $\texttt{M}_{10}$  and $\texttt{M}_{9}^\diamond$ we have ${\bf 0}\neq {\bf 1}$ and so $\texttt{M}_{2}$ and $\texttt{M}_{13}$ imply $\mathcal{P}({\bf 1})$; also  $\texttt{M}_{3}$ (together with ${\bf 0}\neq {\bf 1}$) implies $\forall x (x^k={\bf 0} \longrightarrow x={\bf 0})$  whence $\neg\mathcal{P}({\bf 0})$ holds by $\texttt{M}_{13}$. Then
we note that e.g. the sentence $\forall
x\,\big(x^{2k+1}={\bf 1}\longrightarrow x={\bf 1}\big)$ is derivable from the above axioms, since if $a^{2k+1}=1$ then by $\texttt{M}_{16}$ we have $\mathcal{P}(a^{2k+1})$ or equivalently $\mathcal{P}(a^{2k}\cdot a)$. Now by $\texttt{M}_{13}$ (from which $\mathcal{P}(a^{2k})$ follows) and $\texttt{M}_{15}$ we have $\mathcal{P}(a)$, and so by $\texttt{M}_{13}$,  $a=b^2$ for some  $b$. Now $\texttt{M}_{11,2k+1}$ (since $b^{2\cdot (2k+1)}=1$) implies that either $b=1$ or $b=-1$ holds; in each case we have $a=b^2=1$ (by $\texttt{M}_{11,1}$). Also, the above axioms imply that there are infinitely many elements, since for any $c$ with $c\neq -1,0,1$ (by $\texttt{M}_{9}^\diamond$) we have $c^{k}\neq c^{m}$ whenever $k<m$ (by $\texttt{M}_{11,m-k}$).

\noindent Secondly, the axioms of $\langle\mathbb{R}^+;\times\rangle$ (in Theorem~\ref{thm-r+}) are derivable from the above axioms when they are relativized to $\mathcal{P}$. For example, the relativization of $\texttt{M}_{7,n}^\circ$ which is $\forall x \big(\mathcal{P}(x)\longrightarrow [x^n={\bf 1} \rightarrow x={\bf 1}]\big)$ was actually proved above for odd $n$ (and for even $n$ it follows from $\texttt{M}_{11,n/2}$ noting that  $\texttt{M}_{16}$ and $\texttt{M}_{14}$ imply that $\neg\mathcal{P}({\bf -1})$ holds). We show the relativization of $\texttt{M}_{8,n}$ to $\mathcal{P}$: $\forall x \exists y \big(\mathcal{P}(x)\longrightarrow x=y^n\big)$. Write $n=2^k(2\ell + 1)$; we prove this by induction on $k$. For $k=0$ it follows from $\texttt{M}_{12,\ell}$; for the induction step we note that if $(z)^{2^k(2\ell + 1)}=x$ then we can assume (by $\texttt{M}_{14}$ and $(-z)^{2^k(2\ell + 1)}=(z)^{2^k(2\ell + 1)}$) that $\mathcal{P}(z)$ holds and so the result (the existence of some $y$ with $y^2=z$ whence  $y^{2^{k+1}(2\ell + 1)}=z^{2^k(2\ell + 1)}=x$) follows immediately from $\texttt{M}_{13}$.

\noindent
Finally, the procedure of the quantifier elimination goes as follows. The negations behind $\mathcal{P}$'s can be eliminated by $\texttt{M}_{14}$ which implies (together with $\texttt{M}_{16}$) that $\neg\mathcal{P}(x)\equiv (x={\bf 0}) \vee\mathcal{P}(-x)$. Also by $\texttt{M}_{15}$ we have that $\mathcal{P}(u\cdot v)\equiv[\mathcal{P}(u)\wedge\mathcal{P}(v)]
\vee[\mathcal{P}(-u)\wedge\mathcal{P}(-v)]$. So, we can assume that  $\mathcal{P}(\alpha)$ appears only in the positive form and only when $\alpha$ is either $y$ or $-y$ for a variable $y$. Now, by Lemma~\ref{mainlem}, it suffices to prove the equivalence of the formulas of the form
\begin{equation*}
\exists x ( \mathcal{P}(\diamondsuit x) \wedge  \bigwedge\hspace{-2.5ex}\bigwedge_{i<\ell} x^{n_i}=t_i \;\wedge\;  \bigwedge\hspace{-2.7ex}\bigwedge_{j<k} x^{m_j}\neq s_j )
\end{equation*}
with a quantifier-free formula; where $t_i$'s and $s_j$'s are terms  and $\diamondsuit x$ is either $x$ or $-x$. For each variable $y$ which appears in $t_i$'s or $s_j$'s we have $y={\bf 0}\vee \mathcal{P}(y) \vee \mathcal{P}(-y)$. The case of $y=0$   need not be considered, and by changing $y$ to $-y$ if necessary, we can assume that ``all the variables are positive'', including $x$. Thus, it suffices to eliminate the quantifier of the formula
\begin{equation}\label{r-2}
\exists x ( \mathcal{P}(x) \wedge \bigwedge\hspace{-2.7ex}\bigwedge_{\iota<\alpha} \mathcal{P}(y_\iota) \wedge  \bigwedge\hspace{-2.5ex}\bigwedge_{i<\ell} x^{n_i}=t_i \;\wedge\;  \bigwedge\hspace{-2.7ex}\bigwedge_{j<k} x^{m_j}\neq s_j )
\end{equation}
where all the variables appearing in $t_i$'s and $s_j$'s are among $\{y_\iota\}_{\iota<\alpha}$. Lastly, we can assume that no minus sign ($-$) appears in~(\ref{r-2}) since $-(-u)=u$ and the formulas of the form $v=-w$ can be replaced (are equivalent) with ${\bf 0}\neq {\bf 0}$ (since $v$ and $w$ are positive as all their variables are positive). Now the formula~(\ref{r-2}), when all the variables are positive and no minus sing appears in it, is equivalent with a quantifier-free formula by Proposition~\ref{thm-r+}.
\end{proof}

\begin{theorem}[No Finite Axiomatization for $\langle\mathbb{R};\times\rangle$]\label{thm-infr}
The theory of the structure $\langle\mathbb{R};\times\rangle$ is not finitely axiomatizable.
\end{theorem}
\begin{proof}
For $M=(2N+1)!$ the following set of real numbers     $$\mathbb{R}/M=\{0\}\cup\{(-1)^\iota\prod_{i<\ell}\mathfrak{p}_i^{r_i}
\mid \iota\in\{0,1\},\ell\in\mathbb{N},r_i\in\mathbb{Q}/M\}$$
 with the multiplication operation and the positivity property satisfies  the axioms $$\texttt{M}_1,\texttt{M}_2,\texttt{M}_3,
 \texttt{M}_4,\texttt{M}_9^\diamond,\texttt{M}_{10},
 \texttt{M}_{11,n},\texttt{M}_{13},\texttt{M}_{14},
 \texttt{M}_{15}  \;  (\textrm{for any } n\geqslant 1)$$ and the instances of $\texttt{M}_{12,n}$ for $n=1,2,\cdots,N$, but does not satisfy the instance of $\texttt{M}_{12,n}$ when $2n+1$ is a prime number greater than $M$.
\end{proof}

\section{The Multiplicative Theory of the (Positive) Rational Numbers}
It will be highly fruitful if we have a look at the theory of   $\langle\mathbb{Z};+\rangle$ before axiomatizing $\langle\mathbb{Q}^+;\times\rangle$. Let us note that  the structure $\langle\mathbb{Q}^+;\times,\circ^{-1},{\bf 1}\rangle$ is an abelian group and $\langle\mathbb{Q}^+;\times,\circ^{-1},{\bf 1},<\rangle$ is a regularly dense ordered abelian group (in the terminology of~\cite{robinson-zakon}). The theory of all regularly dense ordered abelian groups is proved to be decidable in~\cite{robinson-zakon}.

\subsection{The Additive Theory of the Integer Numbers}
The theory of the structure $\langle\mathbb{Z};+\rangle$ does not admit quantifier-elimination since for example  the formula $\exists x (a+n\centerdot x = b)$ is not equivalent with a quantifier-free formula (even in the language $\{+,-,{\bf 0},{\bf 1}\}$), where $n\centerdot u={u+\cdots+u} \; [n\!\!-\!\!{\rm times}]$.
However, adding the congruence relations $\{\equiv_{n}\}_{n>1}$ (modulo  standard natural numbers) to the language enables us to prove quantifier-elimination. By definition $a\equiv_n b$ holds when $a-b$ is divisible by (is a multiple of) $n$. For that we use the following version of the generalized Chinese remainder theorem (which is a form of quantifier-elimination).
\begin{proposition}[Generalized Chinese Remainder~\cite{mahler}]\label{crt1}
For  integers $m_0,m_1,\cdots,m_k\geqslant 2$ and $r_0,r_1,\cdots,r_k$ we have
$$\exists x \big(\bigwedge\hspace{-3.65ex}\bigwedge_{0\leqslant i\leqslant k} x\equiv_{m_i} r_i\big) \quad  \iff  \quad  \bigwedge\hspace{-4.75ex}\bigwedge_{0\leqslant i<j\leqslant k} r_i \equiv_{d_{i,j}} r_j$$
where $d_{i,j}$ is the greatest common divisor of $m_i$ and $m_j$ (for each $i<j$).
\end{proposition}
\begin{proof}The `only if' ($\Longrightarrow$) direction is trivial; for the other direction let $p$ be any prime which divides the product $m_0\cdot m_1\cdot\ldots\cdot m_k$, and let $\alpha(i)$ be the greatest number $u$ such that $p^u$ divides $m_i$. Fix an $\ell_p\in\{0,1,\cdots,k\}$ (which depends on $p$) such that $\alpha(\ell_p)$ is the maximum of $\alpha(0),\alpha(1),\cdots,\alpha(k)$. The set of such prime $p$'s is finite. By the (non-generalized) Chinese Remainder Theorem (see e.g. \cite[Chapter~I, Section~6]{smorynski}) there exists some integer $x$ such that $$\bigwedge\hspace{-2.7ex}\bigwedge_{p\in\mathbb{P}} x\equiv_{p^{\alpha(\ell_p)}}r_{\ell_p}.$$
We show that $x\equiv_{m_i}r_i$ holds for any $i$. Fix an $i$; it suffices to show that $x\equiv_{p^{\alpha(i)}}r_i$ holds for any prime $p$ (in the above mentioned finite set). By the definition of $\ell_p$ we have $\alpha(\ell_p)\geqslant\alpha(i)$, so by the assumption $r_{\ell_p}\equiv_{d_{i,\ell_p}}r_i$ we have $r_{\ell_p}\equiv_{p^{\alpha(i)}}r_i$. Thus $x\equiv_{p^{\alpha(\ell_p)}}r_{\ell_p}$ implies $x\equiv_{p^{\alpha(i)}}r_i$.
\end{proof}
\begin{corollary}[An Infinite Version of the Chinese Remainder Theorem]\label{cor-crt}
For  integers $m_0,m_1,\cdots,m_k\geqslant 2$, and $r_0,r_1,\cdots,r_k,n_0,n_1,\cdots,n_\ell$ we have
$$\exists x \big(\bigwedge\hspace{-3.65ex}\bigwedge_{0\leqslant i\leqslant k} x\equiv_{m_i} r_i \;\wedge\; \bigwedge\hspace{-3.5ex}\bigwedge_{0\leqslant \iota\leqslant \ell} x\neq n_\iota\big) \quad  \iff  \quad  \bigwedge\hspace{-4.7ex}\bigwedge_{0\leqslant i<j\leqslant k} r_i \equiv_{d_{i,j}} r_j$$
where $d_{i,j}$ denotes  the greatest common divisor of $m_i$ and $m_j$.
\end{corollary}
\begin{proof}
If the right-hand-side holds then by Proposition~\ref{crt1} there exists some $x_0$ such that $\bigwedge\hspace{-1.85ex}\bigwedge_{0\leqslant i\leqslant k} x_0\equiv_{m_i} r_i$. To make sure that $x_0$ could be taken to be different from all $n_\iota$'s, it suffices to note that for any arbitrarily large  $L$ the number $x=m_0\cdot m_1\cdot\ldots\cdot m_k \cdot L + x_0$  too satisfies $\bigwedge\hspace{-1.85ex}\bigwedge_{0\leqslant i\leqslant k} x\equiv_{m_i} r_i$.
\end{proof}
\begin{proposition}[Infinite Axiomatizablity of $\langle\mathbb{Z};+\rangle$]\label{thm-z}
The following theory
completely axiomatizes the additive theory of the integer numbers and, moreover, the structure $\langle\mathbb{Z};+,-,{\bf 0},{\bf 1},\{\equiv_n\}_{n>1}\rangle$ admits quantifier elimination.
\begin{center}
\begin{tabular}{ll}
($\texttt{A}_1$) \; $\forall x,y,z\,\big(x+(y+z)=(x+y)+z\big)$  & ($\texttt{A}_2$) \; $\forall x\,\big(x+\mathbf{0}=x\big)$ \\
($\texttt{A}_3$) \; $\forall x\,\big(
 x+ (-x)=\mathbf{0}\big)$ & ($\texttt{A}_4$) \; $\forall x,y\,\big(x+y=y+x\big)$ \\
 ($\texttt{A}_{5,n}$) \; $\forall
x\,\big(n\centerdot x={\bf 0}\longrightarrow x={\bf 0}\big)$ & ($\texttt{A}_6^\circ$) \; $ {\bf 1} \neq {\bf 0}$ \\
($\texttt{A}_{7,n}^\circ$) \; $\forall x\exists! y\,\big(\bigvee\hspace{-1.85ex}\bigvee_{i<n}x=n\centerdot y + \bar{i}\,\big)$ & Where $\bar{i}={\bf 1}+\cdots+{\bf 1}$ (\text{for }$i-$\text{times}) \\
and $n > 1$ is a natural number. &
\end{tabular}
\end{center}
\end{proposition}
\begin{proof}
Let us first note that $\texttt{A}_{7,n}^\circ$ is equivalent with  $\forall x\,\big( \underline{\underline{\bigvee}}\hspace{-1.5ex}
\underline{\underline{\bigvee}}_{i<n}
x\equiv_n\bar{i}\,\big)$, where $\underline{\underline{\vee}}$ is the exclusive disjunction; whence
$\underline{\underline{\bigvee}}\hspace{-1.5ex}
\underline{\underline{\bigvee}}_{i}\psi_i \iff \bigvee\hspace{-1.85ex}\bigvee_{i}\psi_i \;\wedge\; \bigwedge\hspace{-1.85ex}\bigwedge_{i\neq j}\neg(\psi_i\wedge \psi_j)$.
So, the negation signs behind the congruences can be eliminated by the equivalences $(t\not\equiv_n u) \iff \bigvee\hspace{-1.85ex}\bigvee_{0<i<n} (t\equiv_n u+\bar{i}\,)$.
By the Main Lemma (\ref{mainlem}) it suffices to eliminate the quantifier of the formula
\begin{equation}\label{r-3}
\exists x ( \bigwedge\hspace{-2.7ex}\bigwedge_{\iota<\alpha} p_\iota\centerdot x\equiv_{q_\iota}r_\iota \;\wedge\; \bigwedge\hspace{-2.5ex}\bigwedge_{i<\ell} n_i\centerdot x=t_i \;\wedge\;  \bigwedge\hspace{-2.7ex}\bigwedge_{j<k} m_j\centerdot x\neq s_j).
\end{equation}
By $\texttt{A}_5$ (and $\texttt{A}_3$) we have $a=b \longleftrightarrow n\centerdot a = n\centerdot b$, and so $a\equiv_m b \longleftrightarrow n\centerdot a \equiv_{mn} n\centerdot b$; whence we can assume that all $p_\iota$'s, $n_i$'s and $m_j$'s in \eqref{r-3} are equal to each other. Thus we show the equivalence of
\begin{equation*}\label{r-4}
\exists x ( \bigwedge\hspace{-2.7ex}\bigwedge_{\iota<\alpha} h\centerdot x\equiv_{q_\iota}r_\iota \;\wedge\; \bigwedge\hspace{-2.5ex}\bigwedge_{i<\ell} h\centerdot x=t_i \;\wedge\;  \bigwedge\hspace{-2.7ex}\bigwedge_{j<k} h\centerdot x\neq s_j)
\end{equation*}
with a quantifier-free formula. But this is equivalent with the following formula (for $y=h\centerdot x$):
\begin{equation}\label{r-5}
\exists y \big( y\equiv_h {\bf 0} \;\wedge\;  \bigwedge\hspace{-2.7ex}\bigwedge_{\iota<\alpha}  y\equiv_{q_\iota}r_\iota \;\wedge\; \bigwedge\hspace{-2.5ex}\bigwedge_{i<\ell} y=t_i \;\wedge\;  \bigwedge\hspace{-2.7ex}\bigwedge_{j<k} y\neq s_j \big).
\end{equation}
Now, if $\ell>0$ then \eqref{r-5} is equivalent with the quantifier-free formula
$$t_0\equiv_h {\bf 0} \;\wedge\;  \bigwedge\hspace{-2.7ex}\bigwedge_{\iota<\alpha}  t_0\equiv_{q_\iota}r_\iota \;\wedge\; \bigwedge\hspace{-2.5ex}\bigwedge_{i<\ell} t_0=t_i \;\wedge\;  \bigwedge\hspace{-2.7ex}\bigwedge_{j<k} t_0\neq s_j$$
 and if $\ell=0$ then \eqref{r-5} is of the form
$$\exists y \big(   \bigwedge\hspace{-2.5ex}\bigwedge_{i<\ell}  y\equiv_{m_i}r_i \;\wedge\ \bigwedge\hspace{-2.7ex}\bigwedge_{j<k} y\neq s_j \big)$$
which is equivalent with a quantifier-free formula by Corollary~\ref{cor-crt}.
\end{proof}

Again this axiomatization of the additive theory of the integer numbers cannot be finite:

\begin{proposition}[No Finite Axiomatization for  $\langle\mathbb{Z};+\rangle$]\label{thm-infz}
The theory of the structures $\langle\mathbb{Z};+\rangle$ is  not finitely axiomatizable.
\end{proposition}
\begin{proof}
Let $\mathfrak{p}$ be an arbitrarily large prime. Trivially, $\mathbb{Z}/\mathfrak{p}$ satisfies $\texttt{A}_1,\texttt{A}_2,\texttt{A}_3,\texttt{A}_4,
\texttt{A}_{5,n}$ and $\texttt{A}_6^\circ$ (for each $n>1$).
However, $\mathbb{Z}/\mathfrak{p}$ does not satisfy $\texttt{A}_{7,n}^\circ$ when $n=k\mathfrak{p}$ is a multiple of $\mathfrak{p}$: if there were some $y\in\mathbb{Z}/\mathfrak{p}$ and $i<k\mathfrak{p}$ such that $1/\mathfrak{p}=k\mathfrak{p}\centerdot y + i$ then $y=b/\mathfrak{p}$ for some $b\in\mathbb{Z}$, and so $1/\mathfrak{p} = kb +i \in\mathbb{Z}$; a contradiction!
But, $\mathbb{Z}/\mathfrak{p}$ satisfies $\texttt{A}_{7,n}^\circ$ when $n<\mathfrak{p}$: since $n$ is relatively prime to $\mathfrak{p}$ then by B\'{e}zout's Lemma there are some $a,b$ such that $an+b\mathfrak{p}=1$. Fix an element $x=z/\mathfrak{p}\in\mathbb{Z}/\mathfrak{p}$ for some $z\in\mathbb{Z}$. By the Division Algorithm there are some $q,i$ such that $bz=nq+i$ and $0\leqslant i<n$. Now, for $y=(az+\mathfrak{p}q)/\mathfrak{p}\in\mathbb{Z}/\mathfrak{p}$ we have $n\centerdot y + i = (naz)/\mathfrak{p} + (n\mathfrak{p}q)/\mathfrak{p} + i = z(an/\mathfrak{p})+bz = z(an+b\mathfrak{p})/\mathfrak{p} =x$.
It can be seen that this $y$ (and also $i$) is unique for $x$. Since if $x=n\centerdot y'+j$ for some $y'=\beta/\mathfrak{p}$ ($\beta\in\mathbb{Z}$) and $0\leqslant j<n$ then for $\alpha=az+\mathfrak{p}q$ we have $\mathfrak{p}(j-i)=n(\alpha-\beta)$ and so by $(n,\mathfrak{p})=1$  the number $n$ should divide $j-i$; thus $j=i$ (because of $0\leqslant i,j<n$) whence $\alpha=\beta$ which implies that $y'=y$.
So, no  finite number of the instances of $\texttt{A}_{7,n}^\circ$ (together with $\texttt{A}_1,\texttt{A}_2,\texttt{A}_3,\texttt{A}_4,
\texttt{A}_{5,n}$ and $\texttt{A}_6^\circ$) can  imply all the instances of $\texttt{A}_{7,n}^\circ$.
\end{proof}

\subsection{The Multiplicative Theory of the Positive Rational Numbers}
We need a version of the Generalized Chinese Remainder Theorem which has more information than Proposition~\ref{crt1} and Corollary~\ref{cor-crt}.
\begin{proposition}[General Chinese Remainder~\cite{ore}]\label{th-crt}
The system $\{x\equiv_{m_i}r_i\}_{i<\ell}$ of congruence equations
has a solution in $\mathbb{Z}$ if and only if for every $i\not=j$,
$r_i\equiv_{d_{i,j}}r_j$ where $d_{i,j}$ is the greatest common
divisor of $m_i$ and $m_j$. Moreover, if $m$ is the least
common multiple of $m_i$'s then the solution $x_0$ {\rm (}if
exists{\rm )} is a linear combination of $r_i(m/m_i)$'s and is  unique modulo  $m$; so all the solutions will be
of the form $L\!\cdot\!m+x_0$ for some (arbitrary)  $L\in\mathbb{Z}$.
\end{proposition}
\begin{proof}
Suppose  $\bigwedge\hspace{-1.85ex}\bigwedge_{i\neq j}r_i\equiv_{d_{i,j}}r_j$; we will
show the existence of an integer  $x_0$ which satisfies $\bigwedge\hspace{-1.85ex}\bigwedge_{i<\ell} x_0\equiv_{m_i}r_i$.
Then of course every number $x=L\!\cdot\!m+x_0$ satisfies the system of equations
 $\bigwedge\hspace{-1.85ex}\bigwedge_{i<\ell} x\equiv_{m_i}r_i$  as well, and every solution $y$ of the equations
  $\bigwedge\hspace{-1.85ex}\bigwedge_{i<\ell} y\equiv_{m_i}r_i$ satisfies $\bigwedge\hspace{-1.85ex}\bigwedge_{i<\ell} x_0\equiv_{m_i}y$, and so
  $x_0\equiv_my$, therefore $y=L\!\cdot\!m+x_0$ for some $L\in\mathbb{N}$.
Since the greatest common divisor of the numbers $\{m/m_i\}_{i<\ell}$ is $1$, by generalized B\'{e}zout's identity,
 there are   $\{c_i\}_{i<\ell}$ such that the identity $\sum_{i<\ell}^{}c_i\!\cdot\!(m/m_i)=1$ holds.
 Let $e_{i,j}$ denote the least common multiplier of $m_i$ and $m_j$. For any $i\neq j$ the number $d_{i,j}$ divides $r_i-r_j$ and the number $e_{i,j}$ divides $m$; so there are $p_{i,j}$ and $q_{i,j}$ such that $r_i-r_j=p_{i,j}\!\cdot\!d_{i,j}$ and $m=q_{i,j}\!\cdot\!e_{i,j}$. By
 $d_{i,j}\!\cdot\!e_{i,j}=m_i\!\cdot\!m_j$ we have
 $(r_i-r_j)m/m_i=p_{i,j}\!\cdot\!d_{i,j}\!\cdot\!q_{i,j}
 \!\cdot\!e_{i,j}/m_i=
p_{i,j}\!\cdot\!q_{i,j}\!\cdot\!m_j$.
Put $x_0=\sum_{i<\ell}r_i\!\cdot\!c_i\!\cdot\!(m/m_i)$. Then
\begin{center}
\begin{tabular}{lll}
$x_0$ & $=$ & $r_jc_jm/m_j+\sum_{i\not=j}r_i\!\cdot\!c_i\!\cdot\!(m/m_i)$ \\
 & $=$ & $r_jc_jm/m_j+\sum_{i\not=j}(r_i-r_j)\!\cdot\!c_i\!\cdot\!(m/m_i)
 +\sum_{i\not=j}r_j\!\cdot\!c_i\!\cdot\!(m/m_i)$ \\
 & $=$ & $r_j\!\cdot\!\sum_ic_i\!\cdot\!(m/m_i)+
 \sum_{i\not=j}(r_i-r_j)\!\cdot\!c_i\!\cdot\!(m/m_i)$\\
 & $=$ & $r_j+\sum_{i\not=j}c_i\!\cdot\!(r_i-r_j)\!\cdot\!(m/m_i)$ \\
  & $=$ & $r_j+\sum_{i\not=j}c_i\!\cdot\!p_{i,j}\!\cdot\!q_{i,j}\!\cdot\!m_j$ \\
    & $=$ & $r_j+m_j\!\cdot\!\sum_{i\not=j}c_i\!\cdot\!p_{i,j}\!\cdot\!q_{i,j},$ \\
         \end{tabular}
\end{center}
which implies the desired conclusion $x_0\equiv_{m_j}r_j$ (for every $j<\ell$).
\end{proof}

The language $\{\times\}$ does not
 allow quantifier elimination for $\langle\mathbb{Q}^+;\times\rangle$, since e.g.
 the formula $\exists x(y=x^2)$ is not equivalent with  a quantifier--free formula. So,
 we introduce the following:
\begin{definition}[The Property of Having the $\mathbf{n}^{\bf th}$ Root]
For any $n\geqslant 2$ let $\Re_n$ be the property of ``being the $n^{\rm th}$ power of a rational number''. In  the other words
  $\Re_n(x)\equiv \exists y[\in\mathbb{Q}]\big(x=y^n\big)$.
\hfill $\oplus\hspace{-1.725ex}\otimes$
\end{definition}

\begin{lemma}[The First Quantifier Elimination for $\boldsymbol\Re$]\label{lem-rn}
The system of relations $\{\Re_{n_i}(u_i \cdot x)\}_{i<\ell}$   has a solution
in $\mathbb{Q}^+$ if and only if $\Re_{d_{i,j}}(u_i\cdot u_j^{-1})$ holds for every $i\not=j$,
 where $d_{i,j}$ is the
greatest common divisor of $n_i$ and $n_j$. Moreover, if
$\bigwedge\hspace{-1.85ex}\bigwedge_{i\not=j}\Re_{d_{i,j}}(u_i\cdot u_j^{-1})$ holds  then
for  $n$ the least common multiplier of $\{n_i\}_{i<\ell}$ and for some fixed $\{c_i\}_{i<\ell}\subseteq\mathbb{Z}$ which satisfy the equality
$\sum_{i<\ell}c_i(n/n_i)=1$,   all of the solutions  are of the form
$w^n\prod_{i<\ell}(u_i)^{-c_i\cdot n/n_i}$ for some (arbitrary)
$w\in\mathbb{Q}^+$.
\end{lemma}
\begin{proof}
Clearly, if $\Re_{n_i}(u_i x)$ and  $\Re_{n_j}(u_j x)$ hold
then $\Re_{d_{i,j}}(u_i x)$ and
$\Re_{d_{i,j}}(u_j^{-1} x^{-1})$, and so
$\Re_{d_{i,j}}(u_i u_j^{-1})$ holds. Conversely, suppose that
$\bigwedge\hspace{-1.85ex}\bigwedge_{i\neq j}\Re_{d_{i,j}}(u_i u_j^{-1})$ holds.    Since the greatest common divisor of
$n/n_i$'s is $1$ there are some $\{c_i\}_{i<\ell}$ such that
$\sum_{i<\ell}c_i (n/n_i)=1$. We show that
$x_0=\prod_{i<\ell}(u_i)^{-c_i n/n_i}$ satisfies
$\bigwedge\hspace{-1.85ex}\bigwedge_{i<\ell}\Re_{n_i}(u_i x_0)$. For  a fix prime ${p}$,  assume the
exponents of ${p}$ in the unique factorizations of
$\{u_i\}_{i<\ell}$  are respectively $\{\alpha_i\}_{i<\ell}$. Then
the exponent of ${p}$ in the unique factorization of $x_0$ will
be $\alpha=\sum_{i<\ell}-c_i\alpha_i(n/n_i)$. Also, by the assumption
$\bigwedge\hspace{-1.85ex}\bigwedge_{i\neq j}\Re_{d_{i,j}}(u_i u_j^{-1})$ we have
$\bigwedge\hspace{-1.85ex}\bigwedge_{i\neq j}\alpha_i\equiv_{d_{i,j}}\alpha_j$. So, by the
proof of the General Chinese Remainder Theorem (Proposition~\ref{th-crt}),
$\bigwedge\hspace{-1.85ex}\bigwedge_i\alpha\equiv_{n_i}-\alpha_i$. This means that the exponent of
(every prime) in the unique factorization of $u_i  x_0$ is a
multiple of $n_i$, whence $\Re_{n_i}(u_i x_0)$ holds (for
each $i<\ell$). Now assume for $y\in\mathbb{Q}^+$ the relation
$\bigwedge\hspace{-1.85ex}\bigwedge_i\Re_{n_i}(u_i y)$ holds. Then for any prime ${p}$, if the exponent of ${p}$ in the unique factorization of $y$
is $\beta$, we have $\bigwedge\hspace{-1.85ex}\bigwedge_i\beta\equiv_{n_i}-\alpha_i$. Whence,
by the  proof of Proposition~\ref{th-crt} we have
$\beta\equiv_n\alpha$, and so
  $y=w^n x_0$   for some $w\in\mathbb{Q}^+$.
\end{proof}

\noindent
Let us note that Lemma~\ref{lem-rn} is a kind of quantifier elimination:
$$\exists x\big[ \bigwedge\hspace{-2.4ex}\bigwedge_{i}\Re_{n_i}(x u_i)\big] \iff \bigwedge\hspace{-2.5ex}\bigwedge_{i\neq j}\Re_{d_{i,j}}(u_i u_j^{-1});$$
so is the next lemma in which we show that
$$\exists x\big[ \bigwedge\hspace{-2.4ex}\bigwedge_{j}\Re_{n_j}(x u_j) \wedge \bigwedge\hspace{-2.4ex}\bigwedge_{k}\!\neg\Re_{m_k}(xv_k)\big] \iff \bigwedge\hspace{-2.5ex}\bigwedge_{i\neq j}\Re_{d_{i,j}}(u_i u_j^{-1}) \wedge \bigwedge\hspace{-4.1ex}\bigwedge_{k:\; m_k \mid
n}\!\!\!\!\!\neg\Re_{m_k}(av_k),$$
where $d_{i,j}$ is the greatest common divisor of $n_i$ and $n_j$, $n$ is the least common multiplier of $\{n_i\}$'s and
$a=\prod_{i<\ell}(u_i)^{-c_in/n_i}$ in which $\sum_{i<\ell}c_in/n_i=1$.

\begin{lemma}[The Second Quantifier Elimination for $\boldsymbol\Re$]\label{lem-rm}
The system of relations $\{\Re_{n_j}(xu_j)\}_{j<\ell},\{\neg
\Re_{m_k}(xv_k)\}_{k<l}$  has a solution (for $x$) in
$\mathbb{Q}^+$ if and only if
for any $i\neq j$ we have $\Re_{d_{i,j}}(u_i u_j^{-1})$ and for any $k$ such that $m_k$ divides $n$ we have $\neg\Re_{m_k}(av_k)$
where $d_{i,j}$ is the greatest common divisor of $n_i$ and $n_j$, $n$ is the least common multiplier of all the $\{n_i\}$'s,
$a=\prod_{i<\ell}(u_i)^{-c_in/n_i}$ and  $\sum_{i<\ell}c_in/n_i=1$.
\end{lemma}
\begin{proof}
Suppose that $x\in\mathbb{Q}^+$ satisfies  the system
$\{\Re_{n_j}(xu_j)\}_{j<\ell},\{\neg \Re_{m_k}(xv_k)\}_{k<l}$. Then by Lemma~\ref{lem-rn},
 $\bigwedge\hspace{-1.85ex}\bigwedge_{i\neq j}
 \Re_{d_{i,j}}(u_iu_j^{-1})$ holds, and moreover
 $x$ is of the form $w^n a$ for some $w\in\mathbb{Q}^+$. We show that
  $\bigwedge\hspace{-1.85ex}\bigwedge_{k:\;m_k\mid n}\!\!\neg \Re_{m_k}(av_k)$ holds too. Suppose
   $m_k\mid  n$. Then $v_k x=v_k w^n a$, and so by $\Re_{m_k}(w^n)$
    and $\neg\Re_{m_k}(v_k x)$ we have  that $\neg \Re_{m_k}(av_k)$.
    Conversely,
    suppose that
    \newline\centerline{$\bigwedge\hspace{-1.85ex}\bigwedge_{i\neq j}\Re_{d_{i,j}}(u_i u_j^{-1})
    \wedge \bigwedge\hspace{-1.85ex}\bigwedge_{k:\;m_k\mid n}\!\!\neg \Re_{m_k}(av_k)$.}
    Then by Lemma~\ref{lem-rn} for any $w\in\mathbb{Q}^+$ the number
    $x=aw^n$ satisfies $\bigwedge\hspace{-1.85ex}\bigwedge_{j<\ell}{\Re_{n_j}(xu_j)}$.
    We choose a suitable $w$ for which $x=aw^n$ also satisfies
    $\bigwedge\hspace{-1.85ex}\bigwedge_{k<l}\neg \Re_{m_k}(xv_k)$. Choose $\mathfrak{p}$ be a (sufficiently large) prime
    number which does not appear in the  (unique) factorization  of any of
     $\{u_j\}_{j<\ell}$ or $\{v_k\}_{k<l}$. Now
     we show that $x=a\mathfrak{p}^n$ satisfies
     $\bigwedge\hspace{-1.85ex}\bigwedge_{k<l}\neg \Re_{m_k}(xv_k)$:

  (i) If $m_k$ divides $n$ then $\neg\Re_{m_k}(av_k)$
and $\Re_{m_k}(\mathfrak{p}^n)$; whence $\neg \Re_{m_k}(xv_k)$.

  (ii) If $m_k$ does not divide $n$, then $\neg\Re_{m_k}(\mathfrak{p}^n)$ and
 so $\neg\Re_{m_k}(xv_k)$, because the prime number $\mathfrak{p}$ does not appear in
 the unique factorization of $a$ or $v_k$
 (if we had $\Re_{m_k}(xv_k)\equiv\Re_{m_k}(a\mathfrak{p}^nv_k)$
 then we must have had $\Re_{m_k}(\mathfrak{p}^n)$ or $m_k\mid n$, a contradiction).
\end{proof}
\begin{corollary}[The Third Quantifier Elimination for $\boldsymbol\Re$]\label{cor-qe}
For any finite sequences $\{t_\iota\}_\iota,\{u_j\}_j,\{v_k\}_k\subseteq\mathbb{Q}^+$ we have
$$\exists x\!\in\!\mathbb{Q}^+\big[ \bigwedge\hspace{-2.5ex}\bigwedge_{\iota} x\neq t_\iota \wedge \bigwedge\hspace{-2.5ex}\bigwedge_{j}\Re_{n_j}(x u_j) \wedge \bigwedge\hspace{-2.5ex}\bigwedge_{k}\!\neg\Re_{m_k}(xv_k)\big] \iff \bigwedge\hspace{-2.55ex}\bigwedge_{i\neq j}\Re_{d_{i,j}}(u_i u_j^{-1}) \wedge \bigwedge\hspace{-4.15ex}\bigwedge_{k:\; m_k \mid
n}\!\!\!\!\!\neg\Re_{m_k}(av_k),$$
where $d_{i,j}$ is the greatest common divisor of $n_i$ and $n_j$, $n$ is the least common multiplier of $\{n_i\}$'s and
$a=\prod_{i}(u_i)^{-c_in/n_i}$ in which $\sum_{i}c_in/n_i=1$.
\end{corollary}
\begin{proof}
It suffices to note that in the proof of Lemma~\ref{lem-rm} there are infinitely many prime
    numbers which do  not appear in the  factorization  of any of
     $\{u_j\}_j$ or $\{v_k\}_k$.
\end{proof}

Now, we have all the necessary tools for proving our desired  quantifier elimination theorem.

\begin{theorem}[The Quantifier Elimination of  $\langle\mathbb{Q}^+;\times,{\bf 1},\circ^{-1},\{\Re_n\}_{n\geqslant 2}\rangle$]\label{th-q1}
The theory of the structure $\langle\mathbb{Q}^+;\times,{\bf 1},\circ^{-1},\Re_2,\Re_3,\cdots\rangle$ admits quantifier elimination.
\end{theorem}
\begin{proof}
By Lemma~\ref{mainlem} it suffices to eliminate the quantifier of the formula
$$\exists x \big[\bigwedge\hspace{-2.5ex}\bigwedge_h x^{\alpha_h}=s_h \wedge
     \bigwedge\hspace{-2.5ex}\bigwedge_i x^{\beta_i}\not=t_i
  \wedge\bigwedge\hspace{-2.5ex}\bigwedge_j \Re_{n_j}(u_j\cdot x^{\gamma_j})
  \wedge\bigwedge\hspace{-2.5ex}\bigwedge_k \neg \Re_{m_k}(v_k\cdot x^{\delta_k}) \big].$$
By the equivalences $a=b \leftrightarrow a^\eta = b^\eta$ and
$\Re_\ell(a) \leftrightarrow \Re_{\ell\eta}(a^\eta)$ we can assume that all the exponents $\alpha_h$'s, $\beta_i$'s, $\gamma_j$'s and $\delta_k$'s are equal, to say $q$. So, we are to eliminate the quantifier of
  $$\exists x \big[\bigwedge\hspace{-2.5ex}\bigwedge_h x^{q}=s_h \wedge
     \bigwedge\hspace{-2.5ex}\bigwedge_i x^{q}\not=t_i
  \wedge\bigwedge\hspace{-2.5ex}\bigwedge_j \Re_{n_j}(u_j\cdot x^{q})
  \wedge\bigwedge\hspace{-2.5ex}\bigwedge_k \neg \Re_{m_k}(v_k\cdot x^{q}) \big]$$
  which is equivalent (for $y=x^q$) with the formula
$$\exists y \big[\Re_q(y)\wedge \bigwedge\hspace{-2.5ex}\bigwedge_h y=s_h \wedge
     \bigwedge\hspace{-2.5ex}\bigwedge_i y\not=t_i
  \wedge\bigwedge\hspace{-2.5ex}\bigwedge_j \Re_{n_j}(u_j  y)
  \wedge\bigwedge\hspace{-2.5ex}\bigwedge_k \neg \Re_{m_k}(v_k y) \big].$$
Thus, it suffices to prove the equivalence of the formulas of the form
\begin{equation}\label{last-f}
\exists x \big[\bigwedge\hspace{-2.5ex}\bigwedge_h x=s_h \wedge
     \bigwedge\hspace{-2.5ex}\bigwedge_i x\not=t_i
  \wedge\bigwedge\hspace{-2.5ex}\bigwedge_j \Re_{n_j}(u_j  x)
  \wedge\bigwedge\hspace{-2.5ex}\bigwedge_k \neg \Re_{m_k}(v_k x) \big]
\end{equation}
with a quantifier-free formula.  If the conjunction $\bigwedge\hspace{-1.85ex}\bigwedge_h x=s_h$ is nonempty then the formula~\eqref{last-f} is equivalent with the quantifier-free formula
$$\bigwedge\hspace{-2.5ex}\bigwedge_h s_0=s_h \wedge
     \bigwedge\hspace{-2.5ex}\bigwedge_i s_0\not=t_i
  \wedge\bigwedge\hspace{-2.5ex}\bigwedge_j \Re_{n_j}(u_js_0)
  \wedge\bigwedge\hspace{-2.5ex}\bigwedge_k \neg \Re_{m_k}(v_k s_0)$$
and otherwise the formula~\eqref{last-f} is actually
$$\exists x \big[     \bigwedge\hspace{-2.5ex}\bigwedge_i x\not=t_i
  \wedge\bigwedge\hspace{-2.5ex}\bigwedge_j \Re_{n_j}(u_j  x)
  \wedge\bigwedge\hspace{-2.5ex}\bigwedge_k \neg \Re_{m_k}(v_k x) \big]$$
which is equivalent with a quantifier-free formula by Corollary~\ref{cor-qe}.
\end{proof}
By a close inspection of the proof of Theorem~\ref{th-q1} (and its prerequisites  Proposition~\ref{th-crt}, Lemmas~\ref{lem-rn},\ref{lem-rm} and Corollary~\ref{cor-qe}) we can give an explicit  complete axiomatization for the multiplicative theory of the positive rational numbers.
\begin{theorem}[Infinite Axiomatizability of $\langle\mathbb{Q}^+;\times\rangle$]
The following theory completely axiomatizes the structure $\langle\mathbb{Q}^+;\times,{\bf 1},\circ^{-1}\rangle$.
\begin{center}
\begin{tabular}{ll}
($\texttt{M}_1$) \; $\forall x,y,z\,\big(x\cdot (y\cdot z)=(x\cdot y)\cdot z\big)$  & ($\texttt{M}_2$) \; $\forall x\,\big(x\cdot \mathbf{1}=x\big)$ \\
($\texttt{M}_3^\circ$) \; $\forall x\,\big(x \cdot  x^{-1}=\mathbf{1}\big)$ & ($\texttt{M}_4$) \; $\forall x,y\,\big(x\cdot y=y\cdot x\big)$ \\
 ($\texttt{M}_{7,n}^\circ$) \; $\forall
x\,\big(x^n={\bf 1}\longrightarrow x={\bf 1}\big)$ & ($\texttt{M}_{16,n}$) \; $\forall v_1,\ldots,v_\ell\exists x\forall z \bigwedge\hspace{-1.85ex}\bigwedge_{k=1}^{\ell}\big(x^n\cdot v_k \neq z^{m_k}\big)$ \\
Where $n \geqslant 1$ is a natural number, and
 & none of $m_k$'s divide $n$ (i.e., $m_1\nmid n,\cdots,m_\ell\nmid n$).
\end{tabular}
\end{center}
\end{theorem}
\begin{proof}
Firstly, we note that the axiom $\texttt{M}_{16,n}$ is equivalent with $\forall v_1,\ldots,v_\ell\exists x \bigwedge\hspace{-1.85ex}\bigwedge_{k=1}^{\ell}
\neg\Re_{m_k}(x^nv_k)$, when no $m_k$ divides $n$ (for $k=1,\cdots,\ell$). Secondly, this axiom implies the existence of infinitely many such $x$'s. Since for fixed  $v_1,\cdots,v_\ell$ there is some $x$ such that $\bigwedge\hspace{-1.85ex}\bigwedge_{k=1}^{\ell}
\neg\Re_{m_k}(x^nv_k)$. Now, fix a prime $\mathfrak{p}>n$. Then for $v_1,\cdots,v_\ell,x^{\mathfrak{p}-n}$ and the sequence $m_1,\cdots,m_k,\mathfrak{p}$ (none of which divides $n$) by this axiom there exists some $y$ such that $\bigwedge\hspace{-1.85ex}\bigwedge_{k=1}^{\ell}
\neg\Re_{m_k}(y^nv_k)\wedge\neg\Re_\mathfrak{p}(y^n
x^{\mathfrak{p}-n})$. Then $y=x$ implies   $\neg\Re_\mathfrak{p}(x^nx^{\mathfrak{p}-n})$ or $\neg\Re_\mathfrak{p}(x^{\mathfrak{p}})$ a contradiction; so $y\neq x$. Continuing this way, by induction, if there are $x_1,\cdots,x_m$ such that $\bigwedge\hspace{-1.85ex}\bigwedge_{i\neq j} x_i\neq x_j
$ and  $\bigwedge\hspace{-1.85ex}\bigwedge_{i=1}^{m}
\bigwedge\hspace{-1.85ex}\bigwedge_{k=1}^{\ell}
\neg\Re_{m_k}(x_i^nv_k)$ then by this axiom for the sequence $v_1,\cdots,v_\ell,x_1^{\mathfrak{p}-n},\cdots,x_m^{\mathfrak{p}-n}$ and the numbers $m_1,\cdots,m_\ell,\mathfrak{p},\cdots,\mathfrak{p}$ (none of which divides $n$) there exists some $y$ such that  $\bigwedge\hspace{-1.85ex}\bigwedge_{k=1}^{\ell}
\neg\Re_{m_k}(y^nv_k)\wedge
\bigwedge\hspace{-1.85ex}\bigwedge_{i=1}^{m}
\neg\Re_{\mathfrak{p}}(y^nx_i^{\mathfrak{p}-n})$. Again it can be seen that $y$ cannot be equal to any of $x_i$'s (for $i=1,\cdots,m$) and so could be taken as $x_{m+1}$. Thus, the axiom $\texttt{M}_{16,n}$ implies that for any sequence $\{v_k\}_{k=1}^{\ell}$ of positive rationals and any sequence $\{m_k\}_{k=1}^{\ell}$ of integers none of which divides $n$ there are infinitely many positive rationals $x$ such that for all $k=1,\cdots,\ell$, $x^nv_k$ is not an $m_k$'s power of any rational number. To see  that $\texttt{M}_{16,n}$ is true (in the set of positive rationals) take $x$ to be a prime number that does not appear in the (unique) factorizations of any of $v_k$'s (for $k=1,\cdots,\ell$); cf. the proof of Lemma~\ref{lem-rm}. So, the axiomatization is sound. To see that it is also complete,  we observe  that the proofs of Theorem~\ref{th-q1}, Lemmas~\ref{lem-rn},\ref{lem-rm} and Corollary~\ref{cor-qe} can go through by using these axioms only, noting that in the proof of Corollary~\ref{cor-qe} and Lemma~\ref{lem-rm} we need the existence of infinitely many $x$'s such that  $\neg\Re_{m_k}(x^n\cdot av_k)$ holds when $m_k$ does not divide $n$ (when $m_k$ divides $n$ then any $x$ satisfies $\neg\Re_{m_k}(x^n\cdot av_k)$ by the assumption $\neg\Re_{m_k}(av_k)$), and this is exactly what the axiom $\texttt{M}_{16,n}$ provides us.
\end{proof}
\begin{theorem}[No Finite Axiomatization for  $\langle\mathbb{Q}^+;\times\rangle$]\label{thm-infq+}
The theory of the structures $\langle\mathbb{Q}^+;\times\rangle$ is  not finitely axiomatizable.
\end{theorem}
\begin{proof}
For any finite number we provide  a model for the axioms $\texttt{M}_{1},\texttt{M}_{2},\texttt{M}_{3}^\circ,
\texttt{M}_{4},\texttt{M}_{16,n}$
and that  finite number of the instances of $\texttt{M}_{7,n}^\circ$ in which some other instances of $\texttt{M}_{7,n}^\circ$ fails. Let $\textswab{p}$ be a sufficiently large prime, and consider   $\{r\cdot\boldsymbol\omega_{\textswab{p}}^k\mid  k\in\mathbb{N}, r\in\mathbb{Q}^+\}$ (together with the multiplication operation). This is a multiplicative subset of the complex numbers whose every member has a unique factorization as $\boldsymbol\omega_{\textswab{p}}^{k}
\prod_{i}\mathfrak{p}_i^{n_i}$
for $0\!\leqslant\!k\!<\!\textswab{p}$ and $n_i\in\mathbb{Z}$. It is straightforward to see that this structure satisfies $\texttt{M}_{1},\texttt{M}_{2},\texttt{M}_{3}^\circ,\texttt{M}_{4}$ and the instances of $\texttt{M}_{7,n}^\circ$ for $n\!<\!\textswab{p}$ but not $\forall
x\,\big(x^{\textswab{p}}={\bf 1}\longrightarrow x={\bf 1}\big)$ since $(\boldsymbol\omega_{\textswab{p}})^{\textswab{p}}=1$ but $\boldsymbol\omega_{\textswab{p}}\neq 1$. It also satisfies $\texttt{M}_{16,n}$ since for any given $v_1,\cdots,v_\ell$ it suffices to take $x$ to be a (sufficiently large) prime that does not appear in the unique factorizations of $v_k$'s.
\end{proof}

\section{Conclusions}
The theory of the multiplication  of the non-negative rational numbers $\langle\mathbb{Q}^{\geqslant 0};\times\rangle$ could also be completely axiomatized by adding the axiom $\forall x\,\big(x\cdot \mathbf{0}=\mathbf{0}=\mathbf{0}^{-1}\big)$ to the axioms of $\langle\mathbb{Q}^{+};\times\rangle$ (and relativizing the axioms $\texttt{M}_3^\circ$ and $\texttt{M}_{16}$ to non-zero elements) just like Proposition~\ref{thm-r0} (with a proof in lines of that of Theorem~\ref{thm-r}). Also by adding the positivity property ($\mathcal{P}(x) \equiv x>0$) to the language we could completely axiomatize the multiplicative theory of the rational numbers $\langle\mathbb{Q};\times\rangle$ (just like Theorem~\ref{thm-r}). Indeed, the theory of the structures $\langle\mathbb{Q};\times,\circ^{-1},{\bf 0},{\bf 1},{\bf -1},\mathcal{P}\rangle$ admits quantifier elimination but unfortunately  $\mathcal{P}$ is not definable in $\langle\mathbb{Q};\times,\circ^{-1},{\bf 0},{\bf 1},{\bf -1}\rangle$. To see this, consider the function from $\mathbb{Q}$ into itself that maps $-1,0,1$ to themselves, and maps each rational number $r$ whose unique factorization is $(-1)^\iota\prod_{j\in\mathbb{N}}\mathfrak{p}_j^{n_j}$ (where cofinitely many of the integers $n_j$'s are zero) to $(-1)^{n_0}r$. This function is  bijective  and preserves the multiplication operation but does not preserve the positivity property, since $\frac{2}{3}$ which is positive is mapped to $-\frac{2}{3}$ which is not positive. We leave open the problem of finding a $\langle\mathbb{Q};\times\rangle-$definable language $\mathcal{L}$ such that $\langle\mathbb{Q};\mathcal{L}\rangle$ admits quantifier elimination. Overall, the theory of the structure $\langle\mathbb{Q};\times\rangle$ is decidable while the theory of the structure $\langle\mathbb{Q};+,\times\rangle$ is not (proved by Robinson~\cite{robinson}); let us note that the decidability of the theories of $\langle\mathbb{R};\times\rangle$ and $\langle\mathbb{C};\times\rangle$ were inherited from the decidability of the theories of $\langle\mathbb{R};+,\times\rangle$ and $\langle\mathbb{C};+,\times\rangle$ by Tarski's results.
Interestingly, the axioms of $\langle\mathbb{R};\times\rangle$ in Theorem~\ref{thm-r} are the laws of signs (positivity and negativity) and multiplication in the high school, and the axioms of $\langle\mathbb{C};\times\rangle$ in Theorem~\ref{thm-c} are the laws learned in the freshmen calculus lessons.
As for the integers, a complete axiomatization, by the method of quantifier elimination, was given for $\langle\mathbb{N}^+;\times\rangle$ in~\cite{cegielski} (see also~\cite{smorynski}). This result can 
be extended to  the theory of the structure $\langle\mathbb{N};\times\rangle$ (cf.~\cite[Exercise~23.17]{monk}). Also, the theory of the structure $\langle\mathbb{Z};\times\rangle$ can be proved to be decidable by the methods of~\cite{cegielski} by providing an explicit (and complete and decidable) axiomatization with the method of quantifier elimination (in a suitable $\{\times\}$-definable  language).

All of the new and old results of the paper are summarized in the following table, in which (only)  the structures $\langle\mathbb{N};+,\times\rangle$, $\langle\mathbb{Z};+,\times\rangle$ and $\langle\mathbb{Q};+,\times\rangle$ are undecidable (while the rest of the structures, $\langle\mathbb{N};+\rangle$, $\langle\mathbb{Z};+\rangle$, $\langle\mathbb{Q};+\rangle$, $\langle\mathbb{R};+\rangle$, $\langle\mathbb{C};+\rangle$, $\langle\mathbb{N};\times\rangle$, $\langle\mathbb{Z};\times\rangle$, $\langle\mathbb{Q};\times\rangle$, $\langle\mathbb{R};\times\rangle$, $\langle\mathbb{C};\times\rangle$, $\langle\mathbb{R};+,\times\rangle$ and $\langle\mathbb{C};+,\times\rangle$,  are decidable and thus axiomatizable by  recursively enumerable sets of sentences).

\begin{table}[h]
\begin{center}
\begin{tabular}{|c||c|c|c|c|c|}
\hline
         & $\mathbb{N}$ & $\mathbb{Z}$ & $\mathbb{Q}$ & $\mathbb{R}$ &
         $\mathbb{C}$ \\
\hline \hline  $\{+\}$ & \cite[Th.~32E]{enderton} &
        Prop.~\ref{thm-z}
        & Prop.~\ref{thm-a} & Prop.~\ref{thm-a}
        & Prop.~\ref{thm-a} \\
 \hline
        $\{\times\}$ & \cite{cegielski} & \cite{cegielski} (\& \S~5)
        & \!\!\!Th.~\ref{th-q1} (\& \S~5)\!\!\! & Th.~\ref{thm-r}  &
        Th.~\ref{thm-c} \\
  \hline\hline $\{+,\times\}$ &
              \!\!\!\cite[Cor.~35A]{enderton}\!\!\! & \cite[Th.~16.7]{monk}
        & \cite{robinson} & \!\!\!\cite[Th.~21.36]{monk}\!\!\! &
        \!\!\!\cite[Th.~21.9]{monk}\!\!\!  \\
 \hline
    \end{tabular}
\end{center}
\end{table}

\bibliographystyle{fundam}


\end{document}